\documentclass[12pt,a4paper]{article}%
\usepackage{amsmath, amsfonts, amsthm,color, latexsym, amssymb}
\usepackage{amscd}
\usepackage{amsmath}
\usepackage{amsfonts}
\usepackage[english]{babel}
\usepackage{amssymb, amsmath, amsfonts, mathrsfs,amsthm, color}
\usepackage{amsmath}
\usepackage{amsfonts}
\usepackage{amssymb}
\usepackage{graphicx}
\usepackage{amssymb}
\usepackage{graphicx}%
\providecommand{\U}[1]{\protect\rule{.1in}{.1in}}
\newtheorem{theorem}{Theorem}[section]

\newtheorem{corollary}[theorem]{Corollary}
\newtheorem{example}[theorem]{Example}

\newtheorem{remark}[theorem]{Remark}

\newtheorem{definition}[theorem]{Definition}
\begin{document}

%
%

\title{ Maximal spaceability in topological vector spaces}
\author{G. Botelho\thanks{Supported by CNPq Grant
306981/2008-4, Fapemig Grant PPM-00295-11 and INCT-Matem\'atica.}\,\,,~D. Cariello,  V. V.
F\'{a}varo\thanks{Supported by FAPEMIG Grant CEX-APQ-00208-09.}~~and D. Pellegrino\thanks{Supported by INCT-Matem\'{a}tica, CNPq Grant 301237/2009-3. \newline 2010 Mathematics Subject Classification: 46B45, 46A16, 46A45.\newline Keywords: spaceability, lineability, topological vector spaces, sequence spaces}}

\date{}
\maketitle

\begin{abstract}
In this paper we introduce a new technique to prove the existence of closed subspaces of maximal dimension inside sets of topological vector sequence spaces. The results we prove cover some sequence spaces not studied before in the context of spaceability and settle some questions on classical sequence spaces that remained open.
\end{abstract}

\section{Introduction}\label{sec1}
Lineability/spaceability, a recent trend in mathematics, is the search for (closed) infinite dimensional linear subspaces inside certain sets of topological vector spaces that fail to have a linear structure. By maximal spaceability we mean the search for closed subspaces of maximal dimension. More precisely, a subset $A$ of a topological vector space $E$ is said to be:

\medskip

$\bullet$ {\it $\mu$-spaceable}, where $\mu$ is a cardinal number, if $A \cup \{0\}$ contains a closed $\mu$-dimensional subspace of $E$,

\medskip

$\bullet$ {\it Maximal spaceable} if it is $\dim E$-spaceable.

\medskip

Many interesting results in the field have been discovered lately, mainly concerning lineability/spaceability in functions spaces \cite{aron, aron-seoane, aron-gurariy, bayart, bernal, garcia-grecu, garcia-seoane, seoane, gurariy-quarta, seoanelaa}, sequence spaces \cite{aron, arondomingo, cleon, laa, timoney}, spaces of linear operators/homogeneous polynomials \cite{bdp, bmp, timoney, pt, ps}, and even spaces of holomorphic functions \cite{lc}. Nevertheless, maximal spaceability is rarely obtained. Almost all of the known results achieve $\mathfrak{c}$-spaceability at most. The aim of this paper is to introduce a maximal spaceability technique for topological vector sequence spaces.

Given an infinite dimensional Banach space $X$, the technique we introduce gives maximal spaceability in topological vector spaces formed by $X$-valued sequences. Very little is known about maximal spaceability in this context, and even some natural $\mathfrak{c}$-spaceability questions remain open. For example, denoting by $\ell_p^w(X)$ the Banach ($p$-Banach if $0 < p < 1$) space of weakly $p$-summable $X$-valued sequences, it is unknown if $\ell_p^w(X) - \ell_q^w(X)$, $0 < q < p < \infty$, is $\mathfrak{c}$-spaceable. In Section \ref{sec2} we prove a quite general maximal spaceabilty result (cf. Theorem \ref{theo}) in the setting of a broad class of topological vector sequence spaces, which includes metrizable and locally convex spaces formed by vector-valued sequences. Particularizing to concrete spaces, we settle several questions on classical sequences spaces that remained open due to the inadequacy of the known previous techniques, for example, we prove (cf. Corollary \ref{cor}) that $\ell_p^w(X) - \ell_q^w(X)$, and even $\ell_p(X) - \ell_q^w(X)$, are actually maximal spaceable. Moreover, we obtain information about sequence spaces that had not been studied before in the context of spaceability, such as $Rad(X)$ and $\ell_p\langle X \rangle$ (cf. Corollaries \ref{corcohen} and \ref{corrad}), both of them important tools in the geometry of Banach spaces.

In order to get information about so many different spaces as particular cases of one single result, an abstract framework -- called {\it sequence functors} (cf. Definition \ref{def}) -- had to be constructed. Although a bit involved at first glance, we believe the richness of examples and the applicability of the main result justify the introduction of such an abstract concept.

In a final remark (cf. Remark \ref{remark}) we point out that the spaceability technique we introduce is not restricted to sequence spaces.

Let $\mathbb{K}=\mathbb{R}$ or $\mathbb{C}$. Henceforth all linear spaces are over $\mathbb{K}$.

\section{Sequence functors}\label{sec2}

Given sets $A_1, A_2$, consider the projections
$$\pi_j \colon A_1 \times A_2 \longrightarrow A_j ~,~\pi_j(a_1, a_2) = a_j~, ~j = 1,2.$$
 Given a linear space $X$, a vector $x \in X$ and a scalar sequence $a =(a_n)_{n=1}^\infty \in \mathbb{K}^{\mathbb{N}}$, consider the $X$-valued sequence
 $$a \otimes x := (a_nx)_{n=1}^\infty. $$
 By $\mathcal{B}$ we mean the class of all Banach spaces over $\mathbb{K}$, by $\mathcal{TVS}$ the class of all topological vector spaces over $\mathbb{K}$ and by ${\cal MSN}$ the class of all metrics and family of seminorms that induce the topology of some space belonging to $\cal TVS$.

 \begin{definition}\rm \label{def}A {\it sequence functor} is a correspondence
$$E \colon \Gamma_E \times {\cal B} \longrightarrow {\mathcal TVS} \times {\cal MSN}, $$
where $\Gamma_E$ is a non-empty set, such that:\\
(a) One of the two following possibilities occurs:\\
(a1) For all $p \in \Gamma_E$ and $X \in {\cal B}$, $\pi_1(E(p,X))$ is a metrizable topological vector space and, in this case, $\pi_2(E(p,X))$ is a translation invariant metric, denoted by $d_{E(p,X)}$, that induces the linear topology of $\pi_1(E(p,X))$; or\\
(a2) For all $p \in \Gamma_E$ and $X \in {\cal B}$, $\pi_1(E(p,X))$ is a locally convex space and, in this case, $\pi_2(E(p,X))$ is a family of seminorms that induces the topology of $\pi_1(E(p,X))$.\\
\indent For simplicty -- and unambiguously -- from now on we shall write, in both cases, $E(p,X)$ for the space, that is, $E(p,X) := \pi_1(E(p,X))$.\\
(b) For all $p \in \Gamma_E$ and $X \in {\cal B}$, $E(p,X)$ is a linear subspace of $ X^{\mathbb{N}}$ with the usual coordinatewise sequence algebraic operations and, given $a \in \mathbb{K}^\mathbb{N}$ and $x \in X$, $x \neq 0$,
$$a \otimes x \in E(p,X) {\rm ~if~and~only~if~} a \in E(p, \mathbb{K}). $$
(c) Let $p \in \Gamma_E$ and an infinite dimensional $X \in {\cal B}$ be given.\\
(c1) If (a1) occurs, there are constants $L_{X,p},M_{X,p} >0$ such that
 $$L_{X,p} \cdot \|x\| \cdot d_{E(p,\mathbb{K})}(a,0) \leq d_{E(p,X)}(a \otimes x,0) \leq M_{X,p} \cdot \|x\| \cdot d_{E(p,\mathbb{K})}(a,0),$$
for all $a \in E(p,\mathbb{K})$ and $x \in X$.\\
(c2) If (a2) occurs, there are a constant  $K_{X,p} > 0$ and a continuous seminorm $\alpha_{E(p,X)}\in \pi_2(E(p,X))$ such that
$$\|w_j\|_X \leq K_{X,p} \cdot \alpha_{E(p,X)}\left(w\right),$$ for all $w = (w_n)_{n=1}^\infty \in E(p,X)$ and $j \in \mathbb{N}$.

When $\Gamma_E$ is a singleton, we shall simply write $E \colon {\cal B} \longrightarrow {\cal TVS} \times {\cal MSN}$, $X \in {\cal B} \mapsto E(X)$.
\end{definition}

\begin{example}\label{exam}\rm (a) The correspondence
$$  X \in {\cal B} \mapsto E(X) = (c_0(X), \|\cdot\|_\infty),$$
where $c_0(X)$ is the Banach space of all norm-null $X$-valued sequences endowed with the sup norm $\|\cdot\|_\infty$, is a sequence functor.

\medskip

\noindent (b) Given a Banach space $X$, let $B_{X}$ denote its closed unit ball and $ X'$ its topological dual. If $0 < p < \infty$, let:\\
 $\bullet$ $\ell_p(X) $ be the Banach ($p$-Banach if $0 < p < \infty$) space of all $X$-valued absolutely $p$-summable sequences $(x_n)_{n=1}^\infty$ (see \cite[Section 8]{df}) endowed with the norm ($p$-norm if $0 < p < 1$)
 $$\|(x_n)_{n=1}^\infty\|_p = \left(\sum_{n=1}^\infty \|x_n\|_X^p \right)^{1/p}. $$
$\bullet$ $\ell_p^w(X) $ be the Banach ($p$-Banach if $0 < p < \infty$) space of all $X$-valued weakly $p$-summable sequences $(x_n)_{n=1}^\infty$, that is, $(\varphi(x_n))_{n=1}^\infty \in \ell_p$ for every $\varphi \in X'$ (see \cite[Section 8]{df}), endowed with the norm ($p$-norm if $0 < p < 1$)
 $$\|(x_n)_{n=1}^\infty\|_{w,p} = \sup_{\varphi \in B_{X'}}\left(\sum_{n=1}^\infty |\varphi(x_n)|^p \right)^{1/p}. $$
 $\bullet$ $\ell_p^u(X) $ be the closed subspace of $\ell_p^w(X) $ formed by all unconditionally $p$-summable sequences $(x_n)_{n=1}^\infty$, that is, $\lim_{k \rightarrow \infty}\|(x_n)_{n=k}^\infty\|_{w,p} = 0$ (see \cite[8.2]{df}).

Since the topologies of Banach and quasi-Banach spaces are induced by translation invariant metrics, the following correspondences are sequence functors:
$$(p,X) \in (0,\infty) \times {\cal B} \mapsto  \left\{\begin{array}{cl} & \hspace*{-1em}E(p,X) = (\ell_p(X), \|\cdot\|_p), \\ & \hspace*{-1em}F(p,X) = (\ell_p^w(X), \|\cdot\|_{w,p}),\\&\hspace*{-1em} G(p,X) = (\ell_p^u(X), \|\cdot\|_{w,p}).  \end{array} \right.$$

\noindent (c) For $1 \leq p < \infty$ and a Banach space $X$, let $\ell_p\langle X \rangle$ be the Banach space of all Cohen-$p$-summable $X$-valued sequences $(x_n)_{n=1}^\infty$ (see \cite{Arregui-Blasco, Bu, Bu-Diestel, Cohen}), that is,
$$\|(x_j)_{j=1}^\infty\|_{\ell_p\langle X\rangle} := \sup\left\{\sum_{j=1}^\infty|y^*_j(x_j)| :
\|(y_j^*)_{j=1}^\infty\|_{\ell^w_{p'}(X')}=1 \right\}<\infty.$$
The correspondence
$$(p,X) \in [1, \infty) \times {\cal B} \mapsto E(p,X) = (\ell_p\langle X \rangle,\|\cdot\|_{\ell_p\langle X\rangle}), $$
is a sequence functor. From \cite[Theorem 1]{Arregui-Blasco}, we have that $\ell_p\langle \mathbb{K} \rangle = \ell_p$, so condition \ref{def}(c1) follows from \cite[Lemma 1]{Arregui-Blasco}.

\medskip

\noindent (d) Given a Banach space $X$, by $Rad(X)$ we mean the Banach space of all almost unconditionally summable $X$-valued sequences $(x_n)_{n=1}^\infty$ (see \cite[p.\,233]{DJT}) endowed with the norm
$$\|(x_n)_{n=1}^\infty\|_{Rad(X)} = \left(\int_0^1 \left\|\sum_{n=1}^\infty r_n(t)x_n \right\|_X^2dt \right)^{1/2}, $$
where $(r_n)_{n=1}^\infty$ are the Rademacher functions. The correspondence
$$X \in  {\cal B} \mapsto E(X) = (Rad(X), \| \cdot \|_{Rad(X)}), $$
is a sequence functor. Khintchine's inequalities (see \cite[1.10]{DJT}) shows that $Rad(\mathbb{K}) = \ell_2$ as sets and that the norm $\|\cdot\|_{Rad(\mathbb{K})}$ is equivalent to the usual $\ell_2$-norm.

\medskip

Now we give two examples where $E(p,\mathbb{K}) \neq \ell_p$; $E(p,\mathbb{K})$ being a non-normed locally convex space in one of them and non-locally convex in the other one.

\medskip

\noindent (e) Given $1 \leq p < \infty$ and a Banach space $X$, consider
$$\ell_p^+(X):= \textstyle\bigcap\limits_{q > p}\ell_q(X),$$
endowed with the locally convex topology generated by the family of norms
$\left(\|\cdot \|_q\right)_{q >p}$. 
 Actually $\ell_p^+(X)$ is a Fr\'echet space. Condition \ref{def}(c2) is easily checked, so the correspondence
$$(p,X) \in [1, \infty) \times {\cal B} \mapsto E(p,X) = \left(\ell_p^+(X), (\|\cdot \|_{q})_{q > p},\right), $$
is a sequence functor. The scalar component $\ell_p^+(\mathbb{K})$ is the Metafune--Moscatelli space $l^{p+}$ \cite{metafune}.

\medskip

\noindent (f) (Lorentz spaces) Let $X$ be a Banach space. For $x=(x_{j})_{j=1}^{\infty}\in\ell_{\infty}(X)$, define
\[
{\mathfrak a}_{X,n}(x):=\inf\left\{  \left\Vert x-u\right\Vert _{\infty}: u\in
c_{00}(X)\text{ and card}(u)<n\right\},
\]
where $c_{00}(X)$ is the space of eventually null $X$-valued sequences, and  card$(u)$ denotes the
cardinality of the set $\{j:u_{j}\neq0\}$ where $u=(u_{j})_{j=1}^\infty\in c_{00}(X)$.

For $0<r,q<+\infty$, the Lorentz sequence space $\ell_{(r,q)}(X)$ (see \cite{mariodaniel}, for the scalar case see \cite{ideals, livropietsch}) consists of all sequences
$x=(x_{j})_{j=1}^{\infty}\in\ell_{\infty}(X)$ such that%
\[
\left(  n^{\frac{1}{r}-\frac{1}{q}}{\mathfrak a}_{X,n}(x)\right)  _{n=1}^{\infty}\in
\ell_{q},
\]
which becomes a complete quasi-normed space with the quasi-norm
\[
\left\Vert x\right\Vert _{\ell_{(r,q)}(X)}=\left\Vert \left(  n^{\frac{1}{r}-\frac{1}%
{q}}{\mathfrak a}_{X,n}(x)\right)  _{n=1}^{\infty}\right\Vert _{q}.
\]
It is well known that $\ell_{(r,q)}(X)\subseteq c_{0}(X)$ and it is also easy to
prove that if $x=(x_{j})_{j=1}^{\infty}\in c_{0}(X),$ then $x$ admits a
non-increasing rearrangement, that is, there is an injection $\pi \colon \mathbb{N}%
\longrightarrow\mathbb{N}$ such that $x_{\pi(1)}\geq x_{\pi(2)}\geq\cdots$ (see \cite[Theorem 13.6.]{ideals}). It follows that ${\mathfrak a}_{X,n}(x) = \|x_{\pi(n)}\|$ for every $n$. Therefore,
$$\|a \otimes y\|_{\ell_{(r,q)}(X)} = \|y\|_X \cdot \|a\|_{\ell_{(r,q)}(\mathbb{K})},  $$
for all $a \in \ell_{(r,q)}(\mathbb{K})$ and $y \in X$.
According to these remarks, the correspondence
$$((r,q),X) \in [(0, \infty)\times (0, \infty)] \times {\cal B} \mapsto E((r,q),X) = \left(\ell_{(r,q)}(X), \|\cdot \|_{\ell_{(r,q)}(X)}\right), $$
is a sequence functor.
\end{example}

\section{Results}\label{sec2}

For a Banach space $X$, given subsets $A$ and $B$ of $X^\mathbb{N}$, by $A - B$ we mean the set $A \cap (X^{\mathbb{N}} - B)$. Let $A$ be a subset of a topological vector space $E$. Unless otherwise explicitly stated, saying that $E - A$ is $\mu$-spaceable means that it is $\mu$-spaceable in $E$. Now we prove our main result:

\begin{theorem}\label{theo} Let $E,F$ be sequence functors and $X$ be an infinite dimensional Banach space. If $\bigcap\limits_{p \in \Gamma_E}E(p,\mathbb{K}) - \bigcup\limits_{q \in \Gamma_F}F(q,\mathbb{K}) \neq \emptyset $, then
$\bigcap\limits_{p \in \Gamma_E} E(p,X) - \bigcup\limits_{q \in \Gamma_F}F(q,X)$
is maximal spaceable in each $E(p,X)$, $p \in \Gamma_E$.
\end{theorem}

\begin{proof}
Let $a = (a_{n})_{n=1}^{\infty}\in\bigcap\limits_{p \in \Gamma_E}E(p,\mathbb{K}) - \bigcup\limits_{q \in \Gamma_F}F(q,\mathbb{K}) $ and observe that $a \neq 0$. Let $p \in \Gamma_E$. As $E(p,X)$ is a sequence functor, the operator
$$T \colon X \longrightarrow E(p,X)~,~T(x) = a \otimes x, $$
is well defined. Its linearity is clear. Let us see that $T$ is injective. Given $x \in X$ with $$0 = T(x) = a \otimes x = (a_nx)_{n=1}^\infty,$$ from the fact that $a_n \neq 0$ for some $n \in \mathbb{N}$ it follows that $x = 0$. So the range of $T$, denoted by $U  = \{a \otimes x : x\in E\}$, is a $\dim X$-dimensional subspace of $E(p,X)$.
%
%
 We just have to show that:\\
\indent (i) $U-\left\{  0\right\}  \subseteq \textstyle\bigcap\limits_{r \in \Gamma_E} E(r,X) - \bigcup\limits_{q \in \Gamma_F}F(q,X)$,\\
\indent (ii) $\dim U=\dim E(p,X)$, \\
\indent (iii) $U$ is a closed subspace of $E(p,X)$.\\ 
(i) Let $u \in U - \{0\}$. Then $u = a \otimes x$ for some $0 \neq x \in X$. Since $a \in\bigcap\limits_{r \in \Gamma_E}E(r,\mathbb{K}) - \bigcup\limits_{q \in \Gamma_F}F(q,\mathbb{K}) $ and $E,F$ are sequence functors, by condition \ref{def}(b) it follows that $u= a \otimes x \in \textstyle\bigcap\limits_{r \in \Gamma_E} E(r,X) - \bigcup\limits_{q \in \Gamma_F}F(q,X)$.   

\noindent(ii) 
We already know that $\dim U =  \dim X$. On the other hand, since $X$ is an infinite dimensional Banach space and $U \subseteq E(p,X) \subseteq X^{\mathbb{N}}$,
$$\dim X = \dim U \leq \dim E(p,X) \leq \dim X^{\mathbb{N}} = \dim X, $$
proving (ii). \\
(iii) Let $w$ belong to the closure of $U$ in $E(p,X)$. Assume first that \ref{def}(a1) occurs. In this metrizable case there is a sequence $ \left(  a \otimes x_k \right)
_{k=1}^{\infty}\subseteq U$ such that $\lim_{k\rightarrow\infty}a \otimes x_k=w$ in $E(p,X).$ In particular, $\left(  a \otimes x_k \right)_{k=1}^{\infty}$ is a Cauchy sequence in $E(p,X)$. As $L_{X,p}, d_{E(p,\mathbb{K})}(a,0)>0$ and the metric $d_{E(p,X)}$ is translation invariant, condition \ref{def}(c1) gives
\begin{align*}\|x_k - x_m\|_X &\leq \frac{1}{L_{X,p} \cdot d_{E(p,\mathbb{K})}(a,0) }\cdot d_{E(p,X)}(a \otimes (x_k - x_m),0) \\
& = \frac{1}{L_{X,p} \cdot d_{E(p,\mathbb{K})}(a,0) }\cdot d_{E(p,X)}(a \otimes x_k - a \otimes x_m,0)\\
& = \frac{1}{L_{X,p} \cdot d_{E(p,\mathbb{K})}(a,0) }\cdot d_{E(p,X)}(a \otimes x_k,  a \otimes x_m) \stackrel{k,m \rightarrow \infty}{\longrightarrow} 0,
\end{align*}
which shows that $(x_k)_{k=1}^\infty$ is a Cauchy sequence in $X$, hence convergent. Say $\lim_{k \rightarrow \infty}x_k = x \in X $. From
\begin{align*}d_{E(p,X)}(a \otimes x_k - a \otimes x,0) &= d_{E(p,X)}(a \otimes (x_k - x),0)\\& \leq M_{X,p} \cdot \|x_k - x\|_X\cdot d_{E(p,\mathbb{K})}(a,0) \stackrel{k \rightarrow \infty}{\longrightarrow}0,
\end{align*}
we conclude that $\lim_{k \rightarrow \infty}(a \otimes x_k - a \otimes x) =0$ in $E(p,X)$. Since the topology on $E(p,X)$ is linear, we have that $\lim_{k \rightarrow \infty}a \otimes x_k = a \otimes x$. It follows that $w = a \otimes x \in U$.

Assume now that \ref{def}(a2) occurs. There is nothing to do if $w =0$, so assume that  $w\neq0$. There are a directed set $\Lambda$ and a net $ \left(  a \otimes x_\lambda \right)
_{\lambda \in \Lambda}\subseteq U$ such that $\lim_{\lambda \in \Lambda}a \otimes x_\lambda=w$ in $E(p,X).$ 
Let $K_{X,p}$ and $\alpha_{E(p,X)}$ be as in condition \ref{def}(c2). Then $\lim_{\lambda \in \Lambda}\alpha_{E(p,X)}(a \otimes x_\lambda-w) =0$. Putting $w = (w_n)_{n=1}^\infty$ and using that the algebraic operations in $E(p,X)$ are the usual coordinatiwise sequence operations,
condition \ref{def}(c2) gives
\begin{align*}\|a_j x_\lambda - w_j\|_X &\leq  K_{X,p} \cdot \alpha_{E(p,X)}\left((a_nx_\lambda - w_n)_{n=1}^\infty\right) \\&=  K_{X,p} \cdot \alpha_{E(p,X)}\left((a_nx_\lambda)_{n=1}^\infty - (w_n)_{n=1}^\infty\right)\\
& =  K_{X,p} \cdot \alpha_{E(p,X)}\left(a \otimes x_\lambda- w\right) \stackrel{\lambda \in \Lambda}{\longrightarrow} 0,
\end{align*}
for every $j \in \mathbb{N}$. So,
\begin{equation}\label{eq1}\lim_{\lambda \in \Lambda}a_{j}x_{\lambda}=w_{j} {\rm ~in~} X {\rm ~ for~ every~} j\in\mathbb{N}.
\end{equation} Since $w\neq0,$
there is $r\in\mathbb{N}$ such that $w_{r}\neq0$. From $\lim_{\lambda \in \Lambda}a_{r}x_{\lambda}=w_{r}\neq 0$ it follows that $a_{r}\neq0$ and $w_{r}=a_{r}\cdot
\lim_{\lambda \in \Lambda}x_{\lambda}$. Calling $x:= \frac{w_r}{a_r}= \lim_{\lambda \in \Lambda}x_{k}\in
X$, from (\ref{eq1}) it follows that $w_{n}=a_{n} x$ for every $n\in\mathbb{N}$. Hence $w = a \otimes x\in U$.
%
%
%
%
%
%
\end{proof}

Considering the sequence functors listed in Example \ref{exam}, Theorem \ref{theo} encompasses several known results as particular cases. But we prefer to focus on some consequences of Theorem \ref{theo} we believe were unknown until now.

To the best of our knowledge, nothing is known about the spaceability of $c_0(X) - \bigcup\limits_{p >0}\ell_p^w(X)$. Since $c_0 - \bigcup\limits_{p > 0}\ell_p \neq \emptyset$, Theorem \ref{theo} gives:

\begin{corollary} Let $X$ be an infinite dimensional Banach space. Then
$$c_0(X) - \textstyle\bigcup\limits_{p >0}\ell_p^w(X)$$
is maximal spaceable.
\end{corollary}

As far as we know, it is unknown whether or not $\ell_p^w(X) - \ell_q^w(X)$ and $\ell_p(X) - \ell_q^w(X)$ are $\mathfrak{c}$-spaceable for $0 < q < p < \infty$. The known techniques, to the best of our
knowledge, are useless in this context. Observe that, as $\ell_p(X)$ is not closed in $\ell_p^w(X)$ for any infinite-dimensional Banach space $X$, the spaceability of $\ell_p^w(X) - \ell_q^w(X)$ is not a straightforward consequence of the spaceability of $\ell_p(X) - \ell_q^w(X)$. Since $\ell_p(\mathbb{K}) - \bigcup\limits_{0<q<p}\ell_q^w(\mathbb{K})  =\ell_p^u(\mathbb{K}) - \bigcup\limits_{0<q<p}\ell_q^w(\mathbb{K})  = \ell_p - \bigcup\limits_{0<q<p}\ell_q \neq \emptyset$, now we know much more:

\begin{corollary}\label{cor} Let $0 < p < \infty$ and $X$ be an infinite dimensional Banach space. Then
$$\ell_p(X) - \textstyle\bigcup\limits_{0 < q <p}\ell_q^w(X){~\rm ~~and~~}\ell_p^u(X) - \textstyle\bigcup\limits_{0 < q <p}\ell_q^w(X)$$
are maximal spaceable. In particular, $\ell_p^w(X) - \textstyle\bigcup\limits_{0 < q <p}\ell_q^w(X)$ is maximal-spaceable.
\end{corollary}


Recalling that $\ell_p\langle X \rangle \subseteq \ell_p(X)$ for $p \geq 1$ (see \cite[p.\,739]{Bu}), in the case $p > 1$ we also know that:

\begin{corollary}\label{corcohen} Let $1 < p < \infty$ and $X$ be an infinite dimensional Banach space. Then
$$\ell_p\langle X \rangle - \textstyle\bigcup\limits_{0 < q <p}\ell_q^w(X)$$
is maximal spaceable.
\end{corollary}

\begin{proof} Use that $\ell_p\langle \mathbb{K} \rangle = \ell_p$ \cite[Theorem 1]{Arregui-Blasco}, $\ell_q^w(\mathbb{K}) = \ell_q$ and $\ell_p - \bigcup\limits_{0<q<p}\ell_q \neq \emptyset$.
\end{proof}

The sequence functor $Rad$ of Example \ref{exam}(d) plays a central role in the geometry of Banach spaces. For example, for $1 \leq p \leq 2 \leq q \leq \infty$, a Banach space $X$ has type $p$ if and only if $\ell_p(X) \subseteq Rad(X)$ and has cotype $q$ if and only if $Rad(X) \subseteq \ell_q(X)$ (cf. \cite[Proposition 12.4]{DJT}). Wonder if, in these cases, $Rad(X) - \ell_p(X)$ and $\ell_q(X) - Rad(X)$ are spaceable? For $p \neq 2 \neq q$ we know much more than that:

\begin{corollary}\label{corrad} Let $X$ be an infinite dimensional Banach space. Then $Rad(X) - \bigcup\limits_{0 < p < 2}\ell_p^w(X)$ and $\bigcap\limits_{p >2}\ell_p(X) -Rad(X)$ are maximal spaceable (the latter in the locally convex topology of  $\ell_2^+(X)$).
\end{corollary}

\begin{proof} Use that $Rad(\mathbb{K}) = \ell_2$ as sets (Khintchine's inequalities \cite[1.10]{DJT}), $\ell_p^w(\mathbb{K}) = \ell_p$,  $\ell_2 - \bigcup\limits_{p \in  (0,2)} \ell_p \neq \emptyset $ and $\ell_2^+ - \ell_2 \neq \emptyset$.
\end{proof}

In \cite{cleon} it is proved that, for an infinite-dimensional Banach space $X$, $\ell_p^+(X) - \ell_p(X)$ is maximal spaceable in the locally convex topology of $\ell_p^+(X)$. Since $\ell_p^+(\mathbb{K}) - \ell_p^w(\mathbb{K}) = \ell_p^+ - \ell_p \neq \emptyset$, now we know a bit more:

\begin{corollary} Let $1 \leq p < \infty$ and $X$ be an infinite dimensional Banach space. Then $\ell_p^+( X) -\ell_p^w(X)$ is maximal spaceable in the locally convex topology of $\ell_p^+(X)$.
\end{corollary}

Several results about Lorentz sequence spaces can be stated. We give just one example. As $\ell_{r_2,q_2} - \ell_{r_1,q_1} \neq \emptyset$ whenever $0< r_1 < r_2 < \infty$ (see \cite[2.1.12]{livropietsch}), we have:

\begin{corollary} Let $0< r_1 < r_2 < \infty$. Then
$\ell_{r_2,q_2}(X) - \ell_{r_1,q_1}(X) $
is maximal spaceable for all $0 < q_1, q_2 < \infty$ and every Banach space $X$.
\end{corollary}

\begin{remark}\label{remark}\rm An examination of the proof of Theorem \ref{theo} makes clear that the argument works in spaces other than sequence spaces. Just to illustrate, let us take a tensorial approach in the normed case. By $\alpha$ we denote a reasonable crossnorm (see \cite[Chapter 6]{ryan}). Given Banach spaces $X$ and $Y$, by $X \otimes_\alpha Y$ we mean the normed space $(X \otimes Y, \alpha)$. According to \cite[Proposition 6.1]{ryan}, $\alpha(x \otimes y) = \|x\|_X \cdot \|y\|_Y$ for all $x \in X$ and $y \in Y$. Reasoning as in the proof of Theorem \ref{theo}, we get:

\medskip

{\it Let $X$ and $Y$ be Banach spaces and $A$ be a proper subset of $Y$. Then $Y \otimes_\alpha X - A \underline{\otimes} X$ is $\dim X$-spaceable in $Y \otimes_\alpha X$, where $A \underline{\otimes} X = \{a \otimes x : a \in A {\rm ~and~}x \in X\}$.}

\medskip

This tensorial approach recovers some of the normed cases we studied before. For $p \geq 1$: $c_0(X) = c_0 \widehat\otimes_\varepsilon X$ \cite[Theorem 1.1.11]{resume} and $\ell_p^u(X) = \ell_p \widehat\otimes_\varepsilon X$ \cite[Proposition 3]{Bu}, where $\varepsilon$ is the injective norm; $\ell_p\langle X \rangle = \ell_p \widehat\otimes_\pi X$ \cite[Theorem 1]{Bu-Diestel}, where $\pi$ is the projective norm; $\ell_p(X) = \ell_p \widehat\otimes_{\Delta_p}X$ \cite[8.1]{df}, where $\Delta_p$ is the reasonable crossnorm studied in \cite[Section 7]{df}.
\end{remark}

\bigskip

\noindent [Geraldo Botelho, Daniel Cariello, Vin\'icius V. F\'avaro] Faculdade de Matem\'atica, Universidade Federal de Uberl\^andia, 38.400-902 -- Uberl\^andia -- Brazil, e-mails: botelho@ufu.br, dcariello@famat.ufu.br,
vvfavaro@gmail.com.

\medskip

\noindent [Daniel Pellegrino] Departamento de Matem\'atica,
Universidade Federal da Para\'iba, 58.051-900 - Jo\~ao Pessoa,
Brazil, e-mail: dmpellegrino@gmail.com.

\end{document}